\documentclass[10pt]{article}
\usepackage[a4paper,centering,hscale=0.7,vscale=0.75]{geometry}

\usepackage{lmodern}
\usepackage[T1]{fontenc}

\usepackage{mathtools}
\usepackage{amsthm}

\makeatletter
\def\@endtheorem{\endtrivlist}
\makeatother

\usepackage{amssymb,%
  mathrsfs}

\usepackage{xcolor}
\usepackage{hyperref}

\newtheorem{prop}{Proposition}[section]
\newtheorem{thm}[prop]{Theorem}
\newtheorem{lemma}[prop]{Lemma}
\newtheorem{coroll}[prop]{Corollary}
\newtheorem{defi}[prop]{Definition}
\theoremstyle{remark}
\newtheorem{rmk}[prop]{Remark}
\theoremstyle{definition}
\newtheorem{hypo}[prop]{Hypothesis}

\newcommand{\cA}{\mathscr{A}}
\newcommand{\cB}{\mathscr{B}}
\DeclareMathOperator*{\dom}{\mathsf{D}}

\newcommand{\cF}{\mathscr{F}}
\newcommand{\cG}{\mathscr{G}}
\newcommand{\cL}{\mathscr{L}}
\newcommand{\bN}{\mathbb{N}}
\renewcommand{\P}{\mathbb{P}}
\newcommand{\cP}{\mathscr{P}}
\newcommand{\bQ}{\mathbb{Q}}
\newcommand{\erre}{\mathbb{R}}

\renewcommand{\geq}{\geqslant}
\renewcommand{\leq}{\leqslant}
\newcommand{\embed}{\hookrightarrow}
\DeclareMathOperator*{\id}{id}
\newcommand{\loc}{\textnormal{loc}}

\newcommand{\ind}[1]{{1}_{#1}}

\numberwithin{equation}{section}

\DeclarePairedDelimiter\abs{\lvert}{\rvert}
\DeclarePairedDelimiter\norm{\lVert}{\rVert}

\DeclarePairedDelimiterX\ip[2]{\langle}{\rangle}{#1,#2}

\mathtoolsset{centercolon}

\makeatletter
\def\@maketitle{%
  \newpage
  \begin{center}%
  \let \footnote \thanks
    {\LARGE \bf \@title \par}%
    \vskip 1.5em%
    {\large
      \lineskip .5em%
      \begin{tabular}[t]{c}%
        \@author
      \end{tabular}\par}%
    \vskip 1em%
    {\@date}%
  \end{center}%
  \par
  \vskip 1em}
\makeatother

\makeatletter
\renewenvironment{enumerate}{%
  \ifnum \@enumdepth >3 \@toodeep\else
      \advance\@enumdepth \@ne
      \edef\@enumctr{enum\romannumeral\the\@enumdepth}\list
      {\csname label\@enumctr\endcsname}{\usecounter
        {\@enumctr}\def\makelabel##1{\hss\llap{\upshape##1}}}\fi
}{%
  \endlist
}

\renewenvironment{itemize}{%
  \ifnum\@itemdepth>3 \@toodeep
  \else \advance\@itemdepth\@ne
    \edef\@itemitem{labelitem\romannumeral\the\@itemdepth}%
    \list{\csname\@itemitem\endcsname}%
      {\def\makelabel##1{\hss\llap{\upshape##1}}}%
  \fi
}{%
  \endlist
}

\makeatother

\begin{document}

\title{A note on stochastic semilinear dissipative evolution equations}
\author{Carlo Marinelli\thanks{Department of Mathematics, University
       College London, Gower Street, London WC1E 6BT, United
       Kingdom.}}

\date{\normalsize 20 December 2025}

\maketitle

\begin{abstract}
  Existence and uniqueness of mild solutions to a class of semilinear
  stochastic evolution equations with additive noise is proved. The
  linear part of the drift term is the generator of a compact
  semigroup of contractions, while the nonlinear part is only assumed
  to be the superposition operator associated to a decreasing
  function.
\end{abstract}


\section{Introduction}
\label{sec:intro}
Consider the stochastic evolution equation
\begin{equation}
  \label{eq:Add}
  du + Au\,dt + f(u)\,dt = B\,dW, \qquad u(0) = u_0,
\end{equation}
on which the following assumptions are made. Let \(G\) be a bounded
domain of \(\erre^d\) with smooth boundary. The unbounded linear
operator \(A\) on \(L^2 := L^2(G)\) is the negative generator of a
strongly continuous contraction semigroup \(S\) on \(L^2\) satisfying
suitable compactness properties (see \S\ref{sec:mr} for precise
conditions). %
Assuming that \(f\colon \erre \to \erre\) is a measurable increasing
function, the superposition operator on \(L^0(G)\) associated to \(f\)
is denoted by the same symbol. %
Let \((\Omega,\cF,\P)\) be a probability space endowed with a
filtration \((\cF_t)_{t \in [0,\infty]}\) on which a cylindrical
Wiener process \(W\) with values in a separable Hilbert space \(K\) is
defined. The initial datum \(u_0\) is a \(\cF_0\)-measurable
\(L^2\)-valued random variable. The diffusion coefficient \(B\) is a
strongly measurable and adapted process taking values in the space of
continuous linear operators from \(K\) to \(L^2\), such that the
stochastic convolution
\[
  \int_0^\cdot S(\cdot - s) B(s)\,dW(s)
\]
is bounded on every bounded subset of \(\erre_+ \times G\) with
probability one. This is clearly the case if, for instance, the
stochastic convolution is continuous in time and space, a condition
that has been extensively studied in the literature (see, e.g.,
\cite{BvN:hn,DPZ,KhoSS:reg}). Denoting the Hilbert space of
Hilbert-Schmidt operators from \(K\) to a Hilbert space \(H\) by
\(\cL^2(K;H)\), an elementary sufficient condition for the stochastic
convolution to be continuous in space and time is that \(H\) is
embedded in \(C(\overline{G})\) and \(B\) is a measurable adapted
\(\cL^2(K;H)\)-valued process such that
\[
  B \in L^0(\Omega;L^2_\loc(\erre_+;\cL^2(K;H))),
\]
that is
\[
  \P\biggl( \int_0^T \norm[\big]{B(t)}^2_{\cL^2(K;H)} dt < \infty \biggr) = 1
  \qquad \forall T \in \erre_+.
\]

The problem considered here is the existence of a unique
probabilistically strong solution to the stochastic evolution equation
\eqref{eq:Add} in a properly defined sense (see \S\ref{sec:mr}). %
As is well known, in contrast to the deterministic setting (see, e.g.,
\cite{Bmax, Dor:NLHY}), there is no general well-posedness result for
stochastic evolution equations such as \eqref{eq:Add} under the sole
assumption that \(A+f\) is a (nonlinear) maximal monotone operator on
\(L^2(G)\). Even though this general problem remains open, some
special cases have been shown to be well posed in a suitable sense. %
A particularly interesting result was obtained in \cite{Barbu:lincei},
where existence and uniqueness of a ``pathwise'' mild solution to
\eqref{eq:Add} is proved in the case that \(A=-\Delta\). Such a
solution is in general only an random function that is not necessarily
a measurable stochastic process.  A more general problem has been
considered in \cite{cm:AP18} (as well as in \cite{cm:div3, cm:refi,
  cm:semimg}), assuming that the linear operator \(A\) is
``variational'' and allowing the diffusion coefficient to depend on
the unknown, obtaining \emph{bona fide} measurable processes as
solutions.

The contribution of this short note is to prove that, without assuming
that \(A\) has any variational structure, but, rather, that \(A\) is
the generator of a compact contraction semigroup, \eqref{eq:Add}
admits a unique probabilistically strong solution. The method of proof
consists of the following steps: the original equation is approximated
by simpler equations where the function \(f\) is replaced by its
Yosida approximation. Working pathwise, i.e.\ for \(\omega\) fixed in
an event of full probability, a priori estimates are obtained for the
solutions to the approximating equations, following closely
\cite{Barbu:lincei}. Limits in appropriate topologies of the solutions
to such approximating equations are then shown to solve the original
equation, relying neither on special properties of the Laplacian nor
on a variational setting for the operator \(A\). In fact, in contrast
to \cite{Barbu:lincei} and \cite{cm:AP18}, instead of relying on
compactness arguments of Aubin-Lions-Simon type, a compactness result
for deterministic convolutions is used (cf. \S\ref{ssec:dcc}).

The remaining text is organized as follows: \S\ref{sec:prel} collects
some preliminaries as well as auxiliary results. The notion of
solution as well as the statement of the well-posedness result and its
proof are given in \S\ref{sec:mr}.

\section{Preliminaries and auxiliary results}
\label{sec:prel}

\subsection{Notation}
\label{ssec:not}
Let us set \(\overline{\erre} = \erre \cup \{-\infty,+\infty\}\),
\(\erre_+ := \mathopen[0,+\infty\mathclose[\), and
\(\erre_+^\times := \erre_+ \setminus \{0\}\).  An expression of the
type \(a \lesssim b\) means that there exists 
\(N \in \erre_+^\times\) such that \(a \leq Nb\).
The identity map of any set is denoted by \(\operatorname{id}\) and
the inverse of a graph or a function \(\phi\) is denoted by
\(\phi^\leftarrow\).
If \((X,\cA,\mu)\) is a measure space, the integral of a measurable
function \(\phi\colon X \to \overline{\erre}\) with respect to \(\mu\)
will also be denoted by \(\mu(\phi)\) or even by \(\mu\phi\). Let
\(E\) be a Banach space. If \(\mu\) is a finite measure, the vector
space of \(\mu\)-equivalence classes of \(\cA\)-measurable functions
\(X \to E\), denoted by \(L^0(\mu;E)\) or \(L^0(X;E)\), is always
assumed to be endowed with the topology of convergence in measure,
thus becoming a complete metrizable topological vector space.  If
\(G\) is the bounded domain of \(\erre^d\) of \S\ref{sec:intro} and
\(p \in [1,\infty]\), we shall denote \(L^p(G)\) simply by \(L^p\).
The domain of an unbounded operator \(u\) on \(E\) is denoted by
\(\dom(u)\). The same notation is used if \(u\) is a graph on \(E\),
i.e.\ a subset of \(E \times E\), in which case then
\(\dom(u) = \operatorname{pr}_1(u)\).
Let \(S\) be a strongly continuous semigroup on \(E\). If
\(\phi\colon \erre_+ \to E\) is a measurable function such that
\(S(t-\cdot)\phi \in L^1(0,t;E)\) for every \(t \in \erre_+\), the
deterministic convolution \(S \ast \phi\colon \erre_+ \to E\) is the
function defined by
\[
  S \ast \phi\colon t \mapsto \int_0^t S(t-s)\phi(s)\,ds.
\]
In particular, \(S \ast \phi\) is well defined for every
\(\phi \in L^1_\loc(\erre_+;E)\).

All random variables and stochastic processes are defined on a fixed
filtered probability space
\((\Omega,\cF,(\cF_t)_{t\in[0,+\infty]},\P)\) that is assumed to
satisfy the ``usual'' conditions of right-continuity and completeness.
Similarly to the deterministic case, if \(M\) is an \(E\)-valued
stochastic process for which the stochastic integral
\(\bigl( S(t - \cdot) \cdot M \bigr)_t\) is well defined for every
\(t \in \erre_+\), the stochastic convolution \(S \diamond M\) is the
process defined by
\[
  S \diamond M(t) := \bigl( S(t - \cdot) \cdot M \bigr)_t
  = \int_0^t S(t-s) \,dM(s).
\]

\subsection{Convex analysis and accretive operators}
\label{ssec:coac}
We collect a few basic definitions and facts from convex analysis and
the theory of accretive operators referring to, e.g.,
\cite{Barbu:type, Mor:Conv, rock}, for details and proofs.

Let \(\phi\colon \erre \to \overline{\erre}\) be a function. We shall
say that \(\phi\) is 1-coercive if
\[
  \lim_{\abs{r} \to \infty} \frac{\phi(r)}{\abs{r}} = +\infty.
\]
The subdifferential of \(\phi\) at a point \(x \in \erre\), denoted by
\(\partial\phi(x)\), is the subset of \(\erre\) defined by
\[
  \partial\phi(x) := \bigl\{ \zeta :\; \phi(y) - \phi(x) \geq \zeta(y -
  x) \quad \forall y \in \erre \bigr\}.
\]
The conjugate of \(\phi\) is the (necessarily convex) function
\(\phi^\ast\colon \erre \to \overline{\erre}\) defined by
\[
  \phi^\ast (y) := \sup_{x \in \erre}\,\bigl( xy - \phi(x) \bigr).
\]
If \(\phi\) is convex and such that \(\phi(x) \neq \infty\) for every
\(x \in \erre\), then \(\phi^\ast\) is 1-coercive.
The definition of \(\phi^\ast\) trivially implies that, for every
\(x,y \in \erre\), \(xy \leq \phi(x) + \phi^\ast(y)\). Moreover, one
has \(xy = \phi(x) + \phi^\ast(y)\) if and only if
\(y \in \partial\phi(x)\).

Given a Banach space \(E\), the \emph{bracket}
\([\cdot{,}\cdot]\colon E \times E \to \erre\) is defined by
\[
  [x,y] := \lim_{h \to 0+} \frac{\norm{x+hy} - \norm{x}}{h}
          = \inf_{h \in \erre_+^\times} \frac{\norm{x+hy} - \norm{x}}{h},
\]
i.e.\ \([x,y]\) is the (right) directional derivative of the norm of
\(E\) at \(x\) in the direction \(y\).
A graph \(A \subset E \times E\) is \emph{accretive} if for every
\((x_1,y_1), (x_2,y_2) \in A\) and \(\lambda \in \erre_+^\times\) one
has
\[
\norm{x_1-x_2 + \lambda(y_1-y_2)} \geq \norm{x_1-x_2}.
\]
An accretive graph \(A\) is \(m\)-\emph{accretive} if
\(A + \operatorname{id}\) is surjective.
Let \(A\) be accretive. Then \((x_1,y_1)\), \((x_2,y_2) \in A\)
implies \([x_1-x_2,y_1-y_2] \geq 0\).
Let \(u_0\) be an element of the closure of \(\dom(A)\) in \(E\),
\(f \in L^1(0,T;E)\), and consider equations of the type
\[
  u' + Au \ni f, \qquad u(0) = u_0.
\]
If \(u_i\) is a mild solution with initial condition \(u_{0i}\) and
right-hand side \(f_i\), \(i \in \{1,2\}\), then, for every
\(t \in [0,T]\),
\[
  \norm{u_1(t) - u_2(t)} \leq \norm{u_{01} - u_{02}}
  + \int_0^t \bigl[ u_1(s) - u_2(s), f_1(s) - f_2(s) \bigr]\,ds.
\]
We shall actually use this estimate only in the case where \(A\) is
linear and \(m\)-accretive, in which case the closure of the domain of
\(A\) is \(E\).

\subsection{An energy estimate for mild solutions}
\begin{lemma}
  \label{lm:sqnm}
  Let \(S\) be a contraction semigroup on a Hilbert space \(H\),
  \(v_0 \in H\), \(f \in L^1_\loc(\erre_+;H)\), and
  \(v = S v_0 + S \ast f\). Then
  \[
    \norm{v(t)}^2 \leq \norm{v_0}^2 + 2 \int_0^t \ip{f(s)}{v(s)}\,ds
    \qquad \forall t \in \erre_+.
  \]
\end{lemma}
\begin{proof}
  The negative generator \(-A\) of \(S\) is maximal monotone, and
  \(v\) is the mild solution to
  \[
    y' + Ay = f, \qquad y(0)=v_0.
  \]
  If \(T \in \erre_+\) and \(f \in L^1(0,T;\dom(A))\) then \(v\) is a
  strong solution, hence the integration-by-parts formula yields
  \[
    \norm{v(t)}^2 + 2 \int_0^t \ip{Av(s)}{v(s)}\,ds = \norm{v_0}^2
    + 2\int_0^t \ip{f(s)}{v(s)}\,ds
  \]
  for every \(t \in [0,T]\). As the second term on the left-hand side
  is positive by the monotonicity of \(A\), the claim follows under
  the additional assumption that \(f \in L^1(0,T;\dom(A))\). The proof
  of the general case then follows noting that
  \(\phi \mapsto S \ast \phi\colon L^1(0,T;H) \to C([0,T];H)\) is continuous
  and \(L^1(0,T;\dom(A))\) is dense in \(L^1(0,T;H)\).
\end{proof}

\subsection{A uniform integrability criterion}
The following is a simple generalization of the sufficiency part in
the criterion for uniform integrability by de la Vall\'ee Poussin in
the special case of finite measure spaces.
\begin{prop}
  \label{prop:dlVPext}  
  Let \((X,\cA,\mu)\) be a finite measure space and
  \(F\colon \erre \to \erre\) a measurable function that is bounded
  from below and 1-coercive.  If \(\cG \subset L^0(\mu)\) is such that
  there exists \(N \in \erre_+^\times\) for which
  \[
    \norm[\big]{F(g)}_{L^1(\mu)} < N \qquad \forall g \in \cG,
  \]
  then \(\cG\) is uniformly integrable with respect to \(\mu\).
\end{prop}
\begin{proof}
  It suffices to show that \(\cG\) is bounded in \(L^1(\mu)\) and that
  for every \(\varepsilon \in \erre_+^\times\) there exist
  \(\delta \in \erre_+^\times\) such that, for every measurable set
  \(A\) with \(\mu(A)<\delta\) and every \(g \in \cG\), one has
  \(\mu(\abs{g} \ind{A}) < \varepsilon\).

  Let us first assume that \(F\) assumes only positive values.  Let
  \(M \in \erre_+^\times\). By the assumptions on \(F\) there exists
  \(R \in \erre_+^\times\) such that \(F(r) > M \abs{r}\) for every
  \(r \in \erre\) with \(\abs{r} > R\). Then, for every \(g \in \cG\)
  and \(A \in \cA\),
  \begin{align*}
    \int_A \abs{g}\,d\mu
    &= \int_{A \cap \{\abs{g}\leq R\}} \abs{g}\,d\mu
    + \int_{A \cap \{\abs{g}> R\}} \abs{g}\,d\mu\\
    & \leq R \mu(A)
    + \frac{1}{M} \int F(g)\,d\mu\\
    &\leq R m(A) + \frac{N}{M}.
  \end{align*}
  Choosing \(A = X\) it immediately follows that \(\cG\) is bounded in
  \(L^1(\mu)\).  Let \(\varepsilon \in \erre_+^\times\) and choose
  \(M\) such that \(N/M < \varepsilon/2\). Then for every
  \(A \in \cA\) with \(\mu(A) < \delta := \varepsilon/(2R)\) one has
  \(\mu(\abs{g}\ind{A}) < \varepsilon\), as desired.

  In order to consider the general case, assume that there exists
  \(\alpha \in \erre_+\) such that \(F \geq -\alpha\). Then
  \(F_\alpha := F + \alpha \geq 0\) and
  \(F_\alpha(r)/\abs{r} \to +\infty\) as \(\abs{r} \to \infty\), as
  well as
  \[
    \int \abs{F_\alpha(g)} \,d\mu = \int F_\alpha(g) \,d\mu
    = \int F(g) \,d\mu + \alpha\mu(X),
  \]
  where the right-hand side is bounded uniformly with respect to
  \(g \in \cG\). The conclusion thus follows from the case where \(F\)
  is positive.
\end{proof}

\subsection{A criterion for weak sequential compactness in \(L^1\)}
\label{ssec:wcL1}
The results in this subsection are adapted from
\cite{Bre-mm}. However, the statements are rearranged and the proofs
are slightly different (and the reference is not easily accessible),
so all details are provided.
\begin{lemma}
  \label{lm:MM.ae}
  Let \((X,\cA,\mu)\) be a finite complete measure space, \(H\) a
  Hilbert space, \(A \subset H \times H\) a maximal monotone graph,
  and \((x_n) \subset L^0(\mu;H)\), \((y_n) \subset L^1(\mu;H)\)
  sequences such that \((x_n,y_n) \in A\) \(\mu\)-a.e.\ for every
  \(n \in \bN\). If \(x_n \to x\) in \(L^0(\mu;H)\) and \(y_n \to y\)
  weakly in \(L^1(\mu;H)\), then \((x,y) \in A\) \(\mu\)-a.e.
\end{lemma}
\begin{proof}
  Passing to a subsequence, let us assume that \(x_n \to x\)
  \(\mu\)-a.e.  This comes at no loss of generality because, after the
  relabeling, one still has that \(y_n \to y\) weakly in
  \(L^1(\mu;H)\) and \((x_n,y_n) \in A\) \(\mu\)-a.e.\ for every
  \(n\).
  For \(N \in \bN\), set \(X_N := \{\abs{x} \leq N\}\). Then
  \(X_N \uparrow X\). So it is enough to prove the claim on every
  \(X_N\). Invoking the Severini-Egorov theorem, it is in fact enough
  to prove the claim assuming that \(x_n \to x\) in
  \(L^\infty(\mu;H)\). Let \((w,g) \in A\) and
  \(\phi \in L^\infty_+(\mu)\). Then
  \({\ip{x_n - w}{y_n - g}}_H\phi \geq 0\) \(\mu\)-a.e., hence,
  denoting the duality pairing of \(L^1(\mu;H)\) and
  \(L^\infty(\mu;H)\) by \(\ip{\cdot}{\cdot}\),
  \[
    0 \leq \int_X {\ip{x_n - w}{y_n - g}}_H\phi \,d\mu
    = \ip[\big]{\phi(x_n - w)}{y_n - g},
  \]
  where, as \(x_n\phi \to x\phi\) in \(L^\infty(\mu;H)\),
  \[
    \lim_{n \to \infty} \ip[\big]{\phi(x_n - w)}{y_n - g} =
    \ip[\big]{\phi(x - w)}{y - g} \geq 0,
  \]
  that is,
  \[
    \int_X {\ip{x - w}{y - g}}_H\phi \,d\mu \geq 0.
  \]
  Since \(w,g,\phi\) are arbitrary, this implies that
  \(\ip{x - w}{y - g}_H \geq 0\) \(\mu\)-a.e.\ for every
  \((w,g) \in A\), hence, by maximality of \(A\), \((x,y) \in A\)
  \(\mu\)-a.e.
\end{proof}

\begin{lemma}
  \label{lm:Brz}
  Let \((X,\cA,\mu)\) be a finite complete measure space and
  \(\gamma\) a maximal monotone graph in \(\erre \times \erre\) with
  \(\dom(\gamma) = \erre\) and \((0,0) \in \gamma\). If the sequences
  of functions \((x_n), (y_n) \subset L^0(\mu)\) are such that
  \((x_n,y_n) \in \gamma\) \(\mu\)-a.e.\ for every \(n \in \bN\) and
  \((x_n y_n)\) is bounded in \(L^1(\mu)\), then \((y_n)\) is
  relatively weakly sequentially compact in \(L^1(\mu)\).
  Furthermore, if \(x_n \to x\) in \(L^0(\mu)\), then, denoting a weak
  limit of a subsequence of \((y_n)\) by \(y\), one has
  \((x,y) \in \gamma\) \(\mu\)-a.e.
\end{lemma}
\begin{proof}
  Let \(G\colon \erre \to \erre_+\) be a convex function with
  \(G(0)=0\) such that \(\gamma = \partial G\). The condition
  \(\dom(\gamma) = \erre\) implies that
  \(G^\ast\colon \erre \to \erre_+\), the convex conjugate of \(G\),
  is such that
  \(\lim_{\abs{r} \to +\infty} G^\ast(r)/\abs{r} = +\infty\).
  Recalling that, for any \(x,y \in \erre\), \(xy=G(x)+G^\ast(y)\) if
  and only if \(y \in \partial G(x) = \gamma(x)\), the assumptions
  imply that \((G^\ast(y_n))\) is bounded in \(L^1(\mu)\). Proposition
  \ref{prop:dlVPext} then yields the uniform integrability of
  \((y_n)\), that is thus relatively weakly sequentially compact in
  \(L^1(\mu)\) by the Dunford-Pettis theorem.
  Lemma \ref{lm:MM.ae} finally implies that
  \((x,y) \in \gamma\) \(\mu\)-a.e.
\end{proof}

\begin{coroll}
  \label{cor:Brz}
  Let \((X,\cA,\mu)\) and \(\gamma\) be as in Lemma \ref{lm:Brz}.  If
  the sequences of functions \((x_n) \subset L^0(\mu)\),
  \((y_n) \in L^1(\mu)\) are such that \((x_n,y_n) \in \gamma\)
  \(\mu\)-a.e.\ for every \(n \in \bN\) and
  \[
    \int x_ny_n \,d\mu \lesssim 1 + \int \abs{y_n} \,d\mu,
  \]
  then \((y_n)\) is relatively weakly compact in \(L^1(\mu)\).
  Furthermore, if \(x_n \to x\) in \(L^0(\mu)\) and \(y\) is an
  adherent point of \((y_n)\), then \((x,y) \in \gamma\)
  \(\mu\)-a.e.
\end{coroll}
\begin{proof}
  By a reasoning entirely analogous to the one used in the proof of
  Lemma \ref{lm:Brz}, there exist \(N_1, N_2 \in \erre_+\) such that,
  for every \(n \in \bN\),
  \begin{equation}
    \label{eq:gigien}
    \int G(x_n)\,d\mu + \int G^\ast(y_n)\,d\mu \leq N_1
    + N_2 \int \abs{y_n}\,d\mu.    
  \end{equation}
  Setting \(\Phi := G^\ast - N_2 \abs{\cdot}\), it follows that for
  every \(n \in \bN\)
  \[
    \int \Phi(y_n)\,d\mu < N_1,
  \]
  where \(\Phi\) is continuous, bounded from below, and
  \(1\)-coercive. The proof is completed appealing to Proposition
  \ref{prop:dlVPext}, the Dunford-Pettis criterion, and Lemma
  \ref{lm:MM.ae}.
\end{proof}
\begin{rmk}
  One needs to assume from the onset that \((y_n) \subset L^1(\mu)\),
  rather than just in \(L^0(\mu)\) as in Lemma \ref{lm:Brz}, otherwise
  the right-hand side of \eqref{eq:gigien} may be infinite.
\end{rmk}

\subsection{Weak continuity of vector-valued functions}
If \(E\) is a Banach space with dual \(E'\) and \(T \in \erre_+\),
recall that a function \(f\colon [0,T] \to E\) is called weakly
continuous if, for every \(\zeta \in E'\),
\(\ip{\zeta}{f}\colon [0,T] \to \erre\) is continuous. The space of
such weakly continuous functions, modulo the set of functions equal to
zero a.e.\ on \([0,T]\), is denoted by \(C_w([0,T];E)\).

The proof of the following results can be found in \cite{Strauss} (or
in \cite{cm:semimg}).
\begin{lemma}
  \label{lm:Strauss}
  Let \(T \in \erre_+\). If \(E\) is a reflexive Banach space that is
  densely and continuously embedded in a Banach space \(F\), then
  \[
    C_w([0,T];F) \cap L^\infty(0,T;E) = C_w([0,T];E).
  \]
\end{lemma}

\subsection{On compactness of deterministic convolutions}
\label{ssec:dcc}
An important tool for the proof of the main result is Theorem
\ref{thm:dcc} below. Throughout this subsection \(S\) is a strongly
continuous compact semigroup on a Banach space \(E\) and \(\Gamma\) is
the map
\begin{align*}
  \Gamma\colon L^1_\loc(\erre_+;E) &\longrightarrow C(\erre_+;E)\\
  \phi &\longmapsto S \ast \phi.
\end{align*}
The following version of the Ascoli-Arzel\`a criterion will be used
(cf., e.g., \cite[TG, X, p. 18, corollaire 2]{Bbk}): Let
\(I \subset \erre\) be a compact interval. A set
\(H \subseteq C(I;E)\) is relatively compact if and only if \(H\) is
equicontinuous and pointwise relatively compact in \(E\).
\begin{thm}
  \label{thm:dcc}
  If \(T \in \erre_+^\times\) and \(H\) is a uniformly integrable
  subset of \(L^1(0,T;E)\), then \(\Gamma(H)\) is relatively compact
  in \(C([0,T];E)\).
\end{thm}
\begin{proof}
  In view of the Ascoli-Arzel\`a criterion, it suffices to prove that
  \(K := \Gamma(H)\) is equicontinuous and pointwise relatively
  compact in \(E\) on \([0,T]\).
  Let us prove that \(K\) is pointwise relatively compact in \(E\) on
  \([0,T]\). It is evident that \(K(0)\) is a singleton, hence it is
  compact in (the separated space) \(E\). Let
  \(t \in \mathopen]0,T\mathclose]\) and
  \(a \in \mathopen]0,t\mathclose[\). One has, for every \(f \in H\),
  \begin{align*}
    \Gamma f(t)
    &= \int_0^{t-a} S(t-s)f(s)\,ds + \int_{t-a}^t S(t-s)f(s)\,ds\\
    &= S(a) \int_0^{t-a} S(t-a-s)f(s)\,ds
      + \int_{t-a}^t S(t-s)f(s)\,ds\\
    &= S(a) \Gamma f(t-a) + \int_{t-a}^t S(t-s)f(s)\,ds
  \end{align*}
  Let \(H_1\) be the closure of the balanced convex envelope of \(H\),
  and \(F\) be the vector space generated by \(H_1\), endowed with the
  norm defined by the gauge function of \(H_1\). Then \(H_1\) is the
  closed unit ball of the complete normed space \(F\) (see, e.g.,
  \cite[EVT, II, p. 28]{Bbk}). Moreover, it is a simple matter to show
  that \(H_1\) inherits uniform integrability from \(H\).  Introduce
  the linear continuous operator
  \begin{align*}
    C_a \colon F &\longrightarrow E\\
    f &\longmapsto S(a) \Gamma f(t-a)
  \end{align*}
  Since \(S(a)\) is a compact endomorphism of \(E\) and \(K(t-a)\) is
  a bounded subset of \(E\), it follows that \(C_a\) is a compact
  operator from \(F\) to \(E\). Let us then show that \(C_a\) tends to
  \(\Gamma(t)\) in \(\cL(F;E)\) as \(a \to 0\). In fact,
  \[
    \norm[\bigg]{\int_{t-a}^t S(t-s)f(s)\,ds}_E
    \leq Me^{\alpha T} \int_{t-a}^t {\norm{f(s)}}_E\,ds,
  \]
  hence, by the uniform integrability of \(H_1\), that coincides with
  the unit ball \(B\) of \(F\),
  \[
    \lim_{a \to 0} \norm[\big]{C_a - \Gamma(t)}_{\cL(F;E)}
    = \lim_{a \to 0} \sup_{f \in B}
    \norm[\bigg]{\int_{t-a}^t S(t-s)f(s)\,ds}_E = 0.
  \]
  Recalling that \(\cL^c(F;E)\) is closed in \(\cL(F;E)\) (see, e.g.,
  \cite[TS, III, p. 4, proposition 2]{Bbk}, this implies that
  \(\Gamma(t) \in \cL^c(F;E)\), hence that \(K(t)\) is relatively
  compact in \(E\). Since \(t\) was arbitrary, we conclude that
  \(K = \Gamma(H)\) is pointwise relatively compact in \(E\) on
  \([0,T]\).

  Let us now show that \(K=\Gamma(H)\) is equicontinuous, which will
  conclude the proof. Let \(t \in [0,T]\) and
  \(\varepsilon \in \erre_+^\times\). We are going to show that there
  exists a neighborhood \(U \subseteq [0,T]\) of \(t\) such that
  \(\norm{k(t)-k(s)} < \varepsilon\) for every \(k \in K\) and every
  \(s \in U\).
  For every \(g,k \in K\) and every \(s \in [0,T]\), the triangle
  inequality yields
  \begin{equation}
    \label{eq:Gfth}
    \norm{k(t) - k(s)} \leq \norm{k(t) - g(t)}
    + \norm{g(t) - g(s)}
    + \norm{g(s) - k(s)}.
  \end{equation}
  For every \(t_0,s \in [0,T]\) with \(t_0 \leq s\) and
  \(k = \Gamma f\), \(f \in H\), one has
  \begin{align*}
    k(s) = \Gamma f(s)
    &= \int_0^s S(s-u) f(u)\,du\\
    &= S(s-t_0) \int_0^{t_0} S(t_0-u) f(u)\,du
      + \int_{t_0}^s S(s-u)f(u)\,du\\
    &= S(s-t_0) k(t_0)
      + \int_{t_0}^s S(s-u)f(u)\,du.
  \end{align*}
  Therefore, if \(g = \Gamma\phi\), \(\phi \in H\),
  \[
    k(s) - g(s) = S(s-t_0)(k(t_0) - g(t_0))
    + \int_{t_0}^s S(s-u)(f(u) - \phi(u))\,du.
  \]
  Setting \(N := Me^{\alpha T}\), one has
  \[
    \norm[\big]{S(s-t_0) \bigl( k(t_0) - g(t_0) \bigr)}
    < N \norm[\big]{k(t_0) - g(t_0)}
  \]
  and
  \[
    \norm[\bigg]{\int_{t_0}^s S(s-u) \bigl(%
      f(u) - \phi(u) \bigr)\,du} \leq N \biggl(
    \int_{t_0}^s \norm{f(u)}\,du + \int_{t_0}^s \norm{\phi(u)}\,du \biggr).
  \]
  The uniform integrability of \(H\) implies that there exists
  \(\eta \in \erre_+^\times\), depending on \(\varepsilon\) only, such
  that for every measurable \(B \subseteq [0,T]\) with
  \(\operatorname{Leb}(B)<\eta\) and every \(f \in H\) one has
  \(\operatorname{Leb}(\norm{f}\ind{B}) < \varepsilon / (12N)\) (the
  reason for the choice of this upper bound will be clear soon). Let
  \(U_0\) be an open interval (in the relative topology of \([0,T]\))
  containing \(t\) and \(t_0 \in [0,\inf U_0]\) be such that
  \(\sup U_0 - t_0 < \eta\). Then for every \(f,\phi \in H\) and
  \(s \in U_0\) one has
  \[
    \norm[\bigg]{\int_{t_0}^s S(s-u) \bigl(%
      f(u) - \phi(u) \bigr)\,du} < \frac{\varepsilon}{6}.
  \]
  Since the closure of \(K(t_0)\) in \(E\) is compact by assumption,
  there exists a finite open cover of \(K(t_0)\) consisting of open
  balls with radius \(\varepsilon/(6N)\), centered around points
  \(k_i(t_0) = \Gamma f_i(t_0)\), \(i=1,\ldots,n\). That is, for every
  \(k \in K\) there exists \(i \in \{1,\ldots,n\}\) such that
  \(\norm{k(t_0)-k_i(t_0)} < \varepsilon/(6N)\), hence also, for every
  \(s \in U_0\),
  \[
    \norm{k(s) - k_i(s)} < \frac{\varepsilon}{3},
  \]
  as well as, by \eqref{eq:Gfth}, (note that obviously \(t \in U_0\))
  \[
    \norm{k(t) - k(s)} < \frac23 \varepsilon + \norm{k_i(t)-k_i(s)}.
  \]
  Since the finite collection of continuous functions \((k_i)\) is
  equicontinuous, there exists a neighborhood \(U_1\) of \(t\) such
  that \(\norm{k_i(t) - k_i(s)} < \varepsilon/3\) for every
  \(i=1,\ldots,n\) and every \(s \in U_1\). The desired neighborhood
  \(U\) of \(t\) can then be chosen as \(U_0 \cap U_1\).
\end{proof}

\begin{rmk}
  If \(S\) is just a strongly continuous semigroup on \(E\) and
  \(T \in \erre_+^\times\), \(\Gamma\) is a continuous linear operator
  from \(L^1(0,T;E)\) to \(C([0,T];E)\), hence it is also continuous
  between the same spaces endowed with the respective weak
  topologies. Therefore \(\Gamma\) maps (relatively) weakly compact
  sets of \(L^1(0,T;E)\) to (relatively) weakly compact sets of
  \(C([0,T];E)\). Theorem \ref{thm:dcc} implies that if \(S\) is
  compact, then \(\Gamma\) maps weakly compact subsets of
  \(L^1(0,T;E)\) to compact subsets of \(C([0,T];E)\). In fact, weakly
  compact subsets of \(L^1(0,T;E)\) are necessarily uniformly
  integrable (cf., e.g., \cite{DieRueScha:wc}).
  On the other hand, the compactness of \(S\) does not imply the
  compactness of \(\Gamma\) from \(L^1(0,T;E)\) to \(C([0,T];E)\), as
  the elementary counterexample of the identity semigroup on
  \(E=\erre\) shows.
\end{rmk}

\section{Main result}
\label{sec:mr}
Let \(S\) be the strongly continuous contraction semigroup on \(L^2\)
generated by \(-A\). The following assumption will be assumed to hold
from now on.
\begin{hypo}
  The semigroup \(S\) admits a (unique) extension to a strongly
  continuous contraction semigroup on \(L^1\), that will be denoted by
  the same symbol.
\end{hypo}

The graph of the increasing function \(f\colon \erre \to \erre\)
admits a unique extension to a maximal monotone graph in
\(\erre \times \erre\), that will be denoted by the same symbol. In
particular, the domain of the maximal monotone graph \(f\) is
\(\erre\).

For notational convenience, we shall denote the Lebesgue measure on
\([0,T]\) and \(G_T := [0,T] \times G\) by \(m_1\) and \(m\),
respectively. Moreover, we shall identify a function
\(\phi \in L^1(0,T;L^1)\) with the function \(\varphi \in L^1(G_T)\)
such that \(\varphi(t,\cdot)=\phi(t)\) for a.a.\ \(t \in [0,T]\) (the
functions in \(L^1(G_T)\) satisfying this condition are equal to
\(\varphi\) \(m\)-a.e.).
\begin{defi}
  A measurable adapted \(L^2\)-valued process \(u\) is a
  \emph{\(L^1\)-mild solution} to \eqref{eq:Add} if there exists a
  measurable adapted \(L^1\)-valued process \(\zeta\) with paths
  belonging to \(L^1(0,T;L^1)\) such that \((u,\zeta) \in f\)
  \(\P \otimes m\)-a.e.\ and
  \[
    u + S \ast \zeta = Su_0 + S \diamond (B \cdot W)
  \]
  as an identity of \(L^1\)-valued processes.
\end{defi}
Note that, according to the notational conventions in
\S\ref{ssec:not}, the stochastic convolution
\(S \diamond (B \cdot W)\) is the process defined by
\[
  \bigl[ S \diamond (B \cdot W) \bigr](t) = \int_0^t S(t-s) B(s)\,dW(s)
  \qquad \forall t \in \erre_+.
\]
\begin{thm}
  \label{thm:main}
  Let \(T \in \erre_+^\times\). Assume that the semigroup \(S\) is
  compact on \(L^1\) and that the stochastic convolution is
  pathwise bounded on \([0,T] \times G\). Furthermore, assume that
  \((0,0) \in f\). Then there exists a unique \(L^1\)-mild solution
  \(u\) to \eqref{eq:Add}. Moreover, the process
  \(v := u - S \diamond (B \cdot W)\) has paths in
  \(C([0,T];L^1) \cap C_w([0,T];L^2)\).
\end{thm}
The proof of Theorem \ref{thm:main} will be a consequence of several
intermediate steps.

Let \((f_\lambda)_{\lambda \in \erre_+^\times}\) be the Yosida
approximation of \(f\), defined by
\[
  f_\lambda := \frac{1}{\lambda} \bigl(\id - (\id + \lambda f)^\leftarrow
  \bigr) \qquad \forall \lambda \in \erre_+^\times.
\]
Since \(f_\lambda\colon \erre \to \erre\) is a Lipschitz continuous
function for every \(\lambda \in \erre_+^\times\), the superposition
operator associated to it is Lipschitz continuous on \(L^p(G)\) for
every \(p \in [1,\infty\mathclose[\). Therefore the equation
\[
  du_\lambda + Au_\lambda\,dt + f_\lambda(u_\lambda)\,dt = B\,dW,
  \qquad u(0) = u_0,
\]
admits a unique mild solution \(u_\lambda\), that is a measurable
adapted \(L^2\)-valued process such that
\[
  u_\lambda + S \ast f_\lambda(u_\lambda) = Su_0
  + S \diamond (B \cdot W).
\]
Setting \(z := S \diamond (B \cdot W)\) and
\(v_\lambda := u_\lambda - z\), one has
\begin{equation}
  \label{eq:vlm}
  v_\lambda + \int_0^t S(t-s) f_\lambda(v_\lambda(s)+z(s))\,ds =
  S(t)u_0,
\end{equation}
i.e.\ \(v_\lambda\) is the unique mild solution to the deterministic
evolution equation with random coefficients
\[
  v'_\lambda + Av_\lambda + f_\lambda(v_\lambda+z) = 0,
  \qquad v_\lambda(0)=u_0.
\]
Denoting the norm in \(L^2\) by \(\norm{\cdot}\), Lemma \ref{lm:sqnm}
then implies
\[
  \norm{v_\lambda(t)}^2
  + 2\int_0^t \ip{f_\lambda(v_\lambda + z)}{v_\lambda}\,ds
  \leq \norm{u_0}^2.
\]
Let \(F\colon \erre \to \erre_+\) be a convex function with \(F(0)=0\)
such that \(f = \partial F\), and
\((F_\lambda)_{\lambda \in \erre_+^\times}\) be the Moreau-Yosida
regularization of \(F\), defined by
\[
  F_\lambda(x) := \inf_{y \in \erre} \biggl(
  \frac{1}{2\lambda}\abs{y-x}^2 + F(y) \biggr).
\]
Then \(\partial F_\lambda = f_\lambda\) for
every \(\lambda \in \erre_+^\times\), thus, by definition of
subdifferential,
\[
  f_\lambda(v_\lambda+z)v_\lambda \geq F_\lambda(v_\lambda+z)
  - F_\lambda(z),
\]
hence
\[
  \norm{v_\lambda(t)}^2
  + 2\int_0^t \int_G F_\lambda(v_\lambda+z)
  \leq \norm{u_0}^2 + 2\int_0^t \int_G F_\lambda(z).
\]
Until further notice, let \(\omega \in \Omega\) be arbitrary but
fixed.  Since \(z(\omega)\) is bounded on the bounded set
\([0,T] \times G\), recalling that \(F_\lambda\) converges pointwise
to \(F\) from below as \(\lambda \to 0\) and that the positivity of
\(F\) implies that \(F_\lambda \geq 0\) for every
\(\lambda \in \erre_+^\times\), one infers that the right-hand side,
thus also the left-hand side, is bounded uniformly with respect to
\(t \in [0,T]\) and \(\lambda \in \erre_+^\times\).  The proof of the
next claim is hence evident.
\begin{lemma}
  \label{lm:bdd}
  The following statements are true:
  \begin{itemize}
  \item[(a)] \((v_\lambda)\) is bounded, hence relatively weakly*
    compact, in \(L^\infty(0,T;L^2)\).
  \item[(b)]
    \(\displaystyle \int_0^t \ip{f_\lambda(v_\lambda + z)}{v_\lambda}
    \leq \norm{u_0}^2\) for every \(t \in [0,T]\) and
    \(\lambda \in \erre_+^\times\).
  \item[(c)]
    \(\displaystyle \biggl( \int_{G_T} F_\lambda(v_\lambda+z)
    \biggr)\) is bounded.
  \end{itemize}
\end{lemma}
We are going to use these bounds in conjunction with the weak
compactness criteria of \S\ref{ssec:wcL1} to pass to the limit as
\(\lambda \to 0\) in \eqref{eq:vlm}.
\begin{lemma}
  \label{lm:ztl}
  The family \((f_\lambda(v_\lambda + z))\) is weakly compact in
  \(L^1(G_T)\). In particular, there exist \(\zeta \in L^1(G_T)\) and
  a sequence \(\lambda' \subseteq \lambda\) such that
  \((f_{\lambda'}(v_{\lambda'} + z))\) converges to \(\zeta\) weakly
  in \(L^1(G_T)\).
\end{lemma}
\begin{proof}
  Starting from the relation
  \[
    f_\lambda(v_\lambda + z) \in f \circ (\id + \lambda f)^\leftarrow
    (v_\lambda + z),
  \]
  setting
  \(\xi_\lambda := (\id + \lambda f)^\leftarrow(v_\lambda + z)\), we
  shall obtain a bound for the integral of (a section of)
  \(f(\xi_\lambda) \xi_\lambda\) over \(G_T\). Writing
  \(J_\lambda := (\id + \lambda f)^\leftarrow\) for convenience, the
  inequality
  \[
    f_\lambda(x) (x - J_\lambda x) = \lambda f_\lambda(x)^2 \geq 0
  \]
  implies \(f_\lambda(x) J_\lambda(x) \leq f_\lambda(x)x\) for all
  \(x \in \erre\). Therefore, writing
  \(\zeta_\lambda := f_\lambda(v_\lambda+z)\),
  \[
    \zeta_\lambda \xi_\lambda = f_\lambda(v_\lambda+z)
    J_\lambda(v_\lambda+z) \leq f_\lambda(v_\lambda+z) (v_\lambda+z)
    = \zeta_\lambda (v_\lambda + z),
  \]
  where \((\xi_\lambda,\zeta_\lambda) \in f\) and, by Lemma
  \ref{lm:bdd}(b), there exists a constant \(N \in \erre_+\)
  independent of \(\lambda\) such that
  \[
    \int_{G_T} \zeta_\lambda v_\lambda < N \qquad \forall
    \lambda \in \erre_+^\times.
  \]
  Therefore, denoting the \(L^\infty(G_T)\) norm by
  \(\norm{\cdot}_\infty\),
  \[
    \int_{G_T} \zeta_\lambda \xi_\lambda \leq N + \norm{z}_\infty
    \int_{G_T} \abs{\zeta_\lambda}.
  \]
  Since \(v_\lambda \in L^2(G_T)\), \(z \in L^\infty(G_T)\), and
  \(f_\lambda\) is Lipschitz continuous, one has
  \(\zeta_\lambda \in L^2(G_T) \embed L^1(G_T)\) for every
  \(\lambda\). Corollary \ref{cor:Brz} then implies that
  \((\zeta_\lambda)\) is weakly compact in \(L^1(G_T)\).
\end{proof}

Let us now obtain adherent points of \((v_\lambda)\).
\begin{lemma}
  \label{lm:vlc}
  If \(S\) admits an extension to a compact semigroup on \(L^1\), then
  \((v_\lambda)\) is relatively compact in \(C([0,T];L^1)\).
\end{lemma}
\begin{proof}
  Adopting the same notation as in the proof of Lemma \ref{lm:ztl},
  the Dunford-Pettis theorem implies that the family
  \((\zeta_\lambda)\), being weakly compact in \(L^1(G_T)\), is
  uniformly integrable with respect to \(m\).\footnote{The reasoning
    leading to the weak compactness of \((\zeta_\lambda)\) actually
    already implies that \((\zeta_\lambda)\) is uniformly integrable.}
  It follows easily by Tonelli's theorem that \((\zeta_\lambda)\),
  considered as a family of \(L^1(G)\)-valued functions on \([0,T]\),
  is also uniformly integrable. In view of \eqref{eq:vlm}, the proof
  is complete appealing to Theorem \ref{thm:dcc}.
\end{proof}

Thanks to Lemma \ref{lm:vlc}, there exist \(v \in C([0,T];L^1)\) and a
subsequence \(\lambda'\) of \(\lambda\) such that
\(v_{\lambda'} \to v\) in \(C([0,T];L^1)\). Passing to a further
subsequence \(\lambda'' \subseteq \lambda'\), one has
\(v_{\lambda''} \to v\) for almost every \((t,x) \in G_T\). The second
part of Lemma \ref{lm:Brz} then implies that \((u,\zeta) \in f\) a.e.\
in \(G_T\).

Taking the limit as \(\lambda \to 0\) in \eqref{eq:vlm} yields,
\[
  v(t) + \int_0^t S(t-s) \zeta(s)\,ds = S(t)u_0
  \qquad \forall t \in [0,T],
\]
or, equivalently, \(v\) is a mild solution in \(L^1(G)\) to the
equation
\[
  v' + Av + \zeta = 0, \qquad v(0)=u_0.
\]
Since \(\omega \in \Omega\) was fixed but arbitrary, this construction
provides a function \(v\colon \Omega \times [0,T] \to L^1\) that,
however, may not even be a measurable process, as the sequence
\(\lambda\) and all its subsequences depend on \(\omega\). This
problem can nonetheless be overcome thanks to the arguments to follow,
in particular Corollary \ref{cor:cvg} below.
\begin{prop}
  Let \(v \in C([0,T];L^1)\) and \(\zeta \in L^1(G_T)\) be such that
  \((v+z,\zeta) \in f\) a.e.\ in \(G_T\) and
  \[
    v' + Av + \zeta = 0, \qquad v(0)=u_0.
  \]
  in the mild sense. Then the pair \((v,\zeta)\) is unique.  
\end{prop}
\begin{proof}
  Denoting the bracket in \(L^1\) by \([\cdot{,}\cdot]\), if
  \((v_1,\zeta_1)\) and \((v_2,\zeta_2)\) are two pairs satisfying the
  hypothesis, then, for every \(t \in [0,T]\),
  \[
    \norm{v_1(t) - v_2(t)} + \int_0^t
    [v_1(s)-v_2(s),\zeta_1(s)-\zeta_2(s)]\,ds \leq 0
  \]
  (cf. \S\ref{ssec:coac}).  Since the integrand on the left-hand side
  is positive by accretivity of \(f\), one has \(v_1=v_2\).
  In order to prove uniqueness of \(\zeta\), it suffices to show that
  if \(\eta \in L^1(0,T;L^1)\) is such that \(S \ast \eta = 0\), i.e.\
  \[
    \int_0^t S(t-s)\eta(s)\,ds = 0 \qquad \forall t \in [0,T],
  \]
  then \(\eta = 0\). Introducing the function
  \(F\colon [0,T]^2 \to L^1\) defined by
  \[
    F(t,s) =
    \begin{cases}
      S(t-s)\eta(s), & \text{ if } s \leq t,\\
      0, & \text{ if } s > t,
    \end{cases}
  \]
  let us show that \(F(u,\cdot)=0\) a.e.\ for every \(u \in
  [0,T]\). To this purpose, taking \(u \in [0,T]\) arbitrary but
  fixed, it suffices to show that \(\int_0^t F(u,s)\,ds=0\) for every
  \(t \in [0,T]\). If \(u \leq t\) then
  \[
    \int_0^t F(u,s)\,ds = \int_0^u S(u-s)\eta(s)\,ds = 0,
  \]
  and if \(u > t\) then
  \[
    \int_0^t F(u,s) = \int_0^t S(u-s)\eta(s)\,ds
    = S(u-t) \int_0^t S(t-s)\eta(s)\,ds = 0.
  \]
  We have thus shown that, for every \(t \in [0,T]\),
  \(S(t-s)\eta(s) = 0\) for all \(s \in [0,t]\) outside a set of
  measure zero that may depend on \(t \in [0,T]\).
  Let \(\tau\) be a countable dense subset of \([0,T]\). Then there
  exists a subset \(I\) of \(\mathopen[0,T\mathclose[\) of full
  measure such that \(S(t-s)\eta(s) = 0\) for every \(t \in \tau\) and
  every \(s \in I \cap [0,t]\). Let \(s \in I\) and
  \((t_n) \subset \tau\) a sequence converging to \(s\) from the
  right.\footnote{For instance, one can take \(t_n\) in an interval of
    radius \(1/n\) centered on \(s+2/n\) for \(n\) large enough.} Then
  \(S(t_n-s)\eta(s)=0\) for every \(n\), hence, taking the limit as
  \(n \to \infty\), the strong continuity of \(S\) implies that
  \(\eta(s)=0\). This proves that \(\eta = 0\) almost everywhere,
  i.e.\ that \(\eta = 0\) in \(L^1(0,T;L^1)\).
\end{proof}
\begin{rmk}
  Instead of using the integral inequality for mild solutions
  involving the bracket, one may use a more elementary (but longer)
  argument, proceeding roughly as follows.  Assuming first that
  \(\zeta \in L^1(0,T;\dom(A))\), one takes a smooth approximation
  \(\sigma\) of the sign function and tests the equation on
  \(\sigma(v_1-v_2)\), then passes to the limit showing that
  \(v_1-v_2=0\). Finally, the continuity of \(\zeta \mapsto v\) and
  the density of \(L^1(0,T;\dom(A))\) in \(L^1(0,T;L^1)\) allow to
  remove the assumption on \(\zeta\).
\end{rmk}

\begin{coroll}
  \label{cor:cvg}
  The sequence \((v_\lambda)\) converges to \(v\) in \(C([0,T];L^1)\)
  and the sequence \((\zeta_\lambda)\) converges weakly to \(\zeta\)
  in \(L^1(G_T)\).
\end{coroll}
\begin{proof}
  Let \(E\) be the topological space obtained as the product of
  \(C([0,T];L^1)\) and of \(L^1(G_T)\) endowed with its weak topology.
  If \((w,\eta)\) is an adherent point of
  \((v_\lambda,\zeta_\lambda)\) in \(E\), then, as seen above,
  \((w+z,\eta) \in f\) a.e.\ in \(G_T\) and \(w' + Aw + \eta = 0\) in
  the mild sense, with \(w(0)=u_0\). Therefore \(w=v\) and
  \(\eta=\zeta\). Since \((v_\lambda,\zeta_\lambda)\) is relatively
  compact in \(E\) and has \((v,\zeta)\) as unique adherent point, the
  claim follows recalling that a filter on a compact space converges
  if and only if it has a unique adherent point (see, e.g., \cite[TG,
  I, p. 60]{Bbk}).
\end{proof}

As a final step, we consider the measurability properties of \(v\) and
\(\zeta\) considered as random functions. The convention according to
which \(\omega \in \Omega\) is fixed is thus dropped from now on. In
particular, we shall look at \(v\) as an \(L^1\)-valued function on
\(\Omega \times [0,T]\) and at \(\zeta\) as a real function on
\(\Omega \times [0,T] \times G\).
\begin{lemma}
  The map \(v\colon \Omega \times [0,T] \to L^1\) is a predictable
  continuous process.
\end{lemma}
\begin{proof}
  It follows immediately by Corollary \ref{cor:cvg} that
  \((v_\lambda)\) converges to \(v\) in \(L^1\) pointwise on
  \(\Omega \times [0,T]\), hence \(v\) is a predictable \(L^1\)-valued
  process with continuous trajectories.  The same conclusion can also
  be obtained noting that Corollary \ref{cor:cvg} implies that \(v\)
  is measurable, adapted, and continuous, in particular predictable.
\end{proof}

\begin{lemma}
  There exists a predictable weakly continuous \(L^2\)-valued process
  \(y\) that, as \(L^1\)-valued process, is indistinguishable from
  \(v\).
\end{lemma}
\begin{proof}
  Lemma \ref{lm:bdd} implies that, for every \(\omega \in \Omega\),
  \(v(\omega) \in L^\infty(0,T;L^2)\), hence, by Lemma
  \ref{lm:Strauss}, \(v(\omega) \in C_w([0,T];L^2)\).
  For every \(\omega \in \Omega\), let \(y(\omega)\) be a Lebesgue
  version of \(v(\omega) \in C([0,T];L^1)\) such that
  \(y(\omega) \in C_w([0,T];L^2)\). Then for every
  \(\omega \in \Omega\) and every \(\varphi \in L^\infty\) the real
  functions \(\ip{v(\omega)}{\varphi}\) and
  \(\ip{y(\omega)}{\varphi}\) are continuous and equal almost
  everywhere on \([0,T]\), hence they are equal everywhere on
  \([0,T]\). Therefore the processes \(\ip{v}{\varphi}\) and
  \(\ip{y}{\varphi}\) are equal (i.e.\ indistinguishable) for every
  \(\varphi \in L^\infty\), hence \(y\) is indistinguishable from
  \(v\) as \(L^1\)-valued processes.
  Let \(\psi \in L^2\). There exists a sequence
  \((\psi_n) \subset L^\infty\) such that \(\psi_n \to \psi\) in
  \(L^2\). Since \(\ip{y}{\psi_n} = \ip{v}{\psi_n}\) is predictable
  and \(\ip{y(\omega,t)}{\psi_n} \to \ip{y(\omega,t)}{\psi}\) for
  every \((\omega,t) \in \Omega \times [0,T]\), \(\ip{y}{\psi}\) is
  predictable.  As \(\psi\) is arbitrary, the Pettis measurability
  theorem implies that \(y\) is predictable as an \(L^2\)-valued
  process.
\end{proof}
\begin{rmk}
  The measurability of \(v\) as an \(L^2\)-valued map does not follow
  from the relative weak* compactness of \((v_\lambda(\omega))\) in
  \(L^\infty(0,T;L^2)\) because this only implies that a subsequence
  of \((v_\lambda(\omega))\) depending on \(\omega\) converges to
  \(v(\omega)\) weakly* in \(L^\infty(0,T;L^2)\).
\end{rmk}

Finally, let us establish measurability properties of
\(\zeta\colon \Omega \times [0,T] \times G \to \erre\).  Since
\((\zeta_\lambda)\) is bounded in \(L^1(G_T)\) \(\P\)-a.s., the random
variable
\[
  Z := \Bigl( 1 + \sup_{\lambda} \norm[\big]{\zeta_\lambda}_{L^1(G_T)}
  \Bigr)^{-1}
\]
is strictly positive and bounded by one.  Therefore the probability
measure \(\bQ\) on \(\cF\) defined by
\[
  \bQ := \frac{1}{\bQ Z} Z \cdot \P
\]
is equivalent to \(\P\). Let
\((\omega,t,x) \mapsto \varphi(\omega,t,x) \in L^\infty(\bQ \otimes
m)\). Then \(\varphi(\omega,\cdot,\cdot) \in L^\infty(G_T)\) for
\(\bQ\)-a.e.\ \(\omega \in \Omega\) and, denoting the canonical pairing
between \(L^1(G_T)\) and \(L^\infty(G_T)\) by \(\ip{\cdot}{\cdot}\),
\(\ip{\zeta_\lambda}{\varphi} \to \ip{\zeta}{\varphi}\) \(\bQ\)-a.s.\
and, setting
\(N := \sup_{\lambda} \norm[\big]{\zeta_\lambda}_{L^1(G_T)}\), one has
\(N \in L^1(\bQ)\) and
\[
  \abs[\big]{\ip{\zeta_\lambda}{\varphi}} \leq N
  \norm{\varphi}_{L^\infty(\bQ \otimes m)} \qquad \bQ\text{-a.s.},
\]
hence, by the dominated convergence theorem,
\(\bQ \ip{\zeta_\lambda}{\varphi} \to \bQ \ip{\zeta}{\varphi}\).
Since \(\varphi\) is arbitrary, this shows that
\(\zeta_\lambda \to \zeta\) weakly in \(L^1(\bQ \otimes m)\).
Therefore there exists a sequence \((\widetilde{\zeta}_n)\) of
(finite) convex combinations of \((\zeta_\lambda)\) such that
\(\widetilde{\zeta}_n \to \zeta\) in \(L^1(\bQ \otimes m)\).  Let
\(\cP\) and \(\cB(G)\) denote the predictable \(\sigma\)-algebra and
the Borel \(\sigma\)-algebra of \(G\), respectively. Since
\(\widetilde{\zeta}_n\) is \(\cP \otimes \cB(G)\)-measurable, its
limit \(\zeta\) belongs to
\(L^1(\Omega \times [0,T] \times G, \cP \otimes \cB(G), \bQ \otimes
m)\), in particular \(\zeta\) is measurable with respect to the
\(\bQ \otimes m\)-completion of \(\cP \otimes \cB(G)\). Then
\((\omega,t) \mapsto \zeta(\omega,t,\cdot)\) is an \(L^1\)-valued
function that is measurable with respect to the
\(\bQ \otimes m_1\)-completion of \(\cP\). Therefore, recalling that
\(\bQ\) is equivalent to \(\P\), there exists a
\(\P \otimes m_1\)-version of \(\zeta\), denoted by the same symbol,
that is measurable with respect to \(\cP\).

\smallskip

We have thus completed the proof of Theorem \ref{thm:main}.

\begin{rmk}
  Minor modifications of the proof allow to consider equation
  \eqref{eq:Add} with \(A\) and \(f\) quasi-monotone, rather than
  monotone. This means that there exist constants \(\alpha\),
  \(\beta \in \erre_+\) such that \(A + \alpha\operatorname{id}\) and
  \(f + \beta\operatorname{\id}\) are monotone. Setting
  \(k :=\alpha+\beta\), the problem is equivalent to proving existence
  and uniqueness of an \(L^1\)-mild solution to
  \[
    du + \bigl(Au + f(u) - ku\bigr)\,dt = B\,dW, \qquad u(0) = u_0,
  \]
  with \(A\) and \(f\) monotone. Replacing \(f\) by its Yosida
  approximation \(f_\lambda\), writing the equation in mild form, and
  subtracting the stochastic convolution, one obtains the equation
  \[
    v'_\lambda + Av_\lambda + f_\lambda(v_\lambda + z) = k(v_\lambda + z),
    \qquad v_\lambda(0) = u_0.
  \]
  Then one easily proves that Lemma \ref{lm:bdd} continues to hold,
  after which the rest of the proof can be repeated almost word by
  word.
\end{rmk}

\begin{rmk}
  The assumption \((0,0) \in f\) is not a major restriction. In fact,
  if \(\alpha \in \erre\) is such that \((0,\alpha) \in f\) and the
  graph \(g \subset \erre \times \erre\) is defined by
  \(g(x) := f(x) - \alpha\) for every \(x \in \erre\), so that
  \((0,0) \in g\), the mild form of equation \eqref{eq:Add} reads
  \[
    u + S \ast (g(u) + \alpha) = Su_0 + S \diamond (B \cdot W),
  \]
  or equivalently, denoting the function \((t,x) \mapsto 1\) by \(1\),
  \[
    u + S \ast g(u) = Su_0 + S \diamond (B \cdot W)
    - \alpha S \ast 1.
  \]
  Setting \(z := S \diamond (B \cdot W) - \alpha S \ast 1\), it is
  clear that the whole proof goes through unchanged, provided that
  \(S \ast 1\) is bounded on \(G_T\). This is the case, for instance,
  if \(S\) is sub-Markovian.
\end{rmk}

\begin{rmk}
  The choice of the noise term as a Wiener process is for convenience
  only. Inspection of the proof quickly reveals that the noise
  \(M := B \cdot W\) can be replaced by any other process \(M\) for
  which the stochastic convolution \(S \diamond M\) can be
  meaningfully defined and has paths that are bounded in time and
  space. In general, however, depending on the measurability property
  of the stochatic convolution, one may only obtain measurable
  processes as solutions.
\end{rmk}

\bibliographystyle{plainurl}
\bibliography{ref}

\end{document}